\documentclass[english]{ourlematema_dom}
\usepackage[utf8]{inputenc}
\usepackage{amsmath}
\usepackage{graphicx}
\usepackage{amsfonts}
\usepackage{amssymb}
\usepackage{pdfpages}
\usepackage{comment}
\usepackage{enumitem,pstricks,xy,pst-node}
\xyoption{all}
\usepackage{amsthm}
\usepackage{longtable}
\usepackage{mathtools}
\usepackage{caption,subcaption} 
\usepackage{booktabs} %to produce nicer tables
\usepackage[color]{changebar} %to introduce changebars
\cbcolor{red} %the color of the changebars
\usepackage{mathrsfs} %to use \mathscr
\usepackage{verbatim} %to use \verb
\usepackage[disable] %to do notes in the document
{todonotes}
\usepackage[normalem]{ulem} %to strike out text
\usepackage{faktor} 
\usepackage{caption} %to put some space between tables and caption
\usepackage{nicefrac} %to use nice inline fractions, e.g. 1/2
\usepackage{extarrows} %arrows with labels \xleftarrow{text}
\usepackage{enumitem} % to change indent of itemize
\usepackage{hyperref} % to add urls and hrefs

\newtheorem{theorem}{Theorem}[section]
\newtheorem{proposition}[theorem]{Proposition}

\newtheorem{definition}[theorem]{Definition}
\theoremstyle{definition}

\newtheorem{remark}[theorem]{Remark}
\newtheorem{conjecture}[theorem]{Conjecture}

\theoremstyle{remark}
\newtheorem{example}[theorem]{Example}

\newcommand{\cC}{\mathcal{C}}
\newcommand{\cE}{\mathcal{E}}
\newcommand{\cF}{\mathcal{F}}
\newcommand{\cG}{\mathcal{G}}

\newcommand{\cL}{\mathcal{L}}
\newcommand{\cO}{\mathcal{O}}

\newcommand{\uu}{\mathbf{U}}
\newcommand{\gr}{\mathbf{G}}
\newcommand{\Gr}{\mathbf{Gr}}
\newcommand{\R}{\mathbf{R}}
\newcommand{\PP}{\mathbb{P}}

\DeclareMathOperator{\SL}{SL}
\DeclareMathOperator{\Hb}{H}
\DeclareMathOperator{\hsm}{h}

\DeclareMathOperator{\dt}{det}

\DeclareMathOperator{\length}{length}
\DeclareMathOperator{\cplx}{cplx}
\DeclareMathOperator{\reg}{reg}
\DeclareMathOperator{\syl}{Syl}
\DeclareMathOperator{\res}{Res}
\DeclareMathOperator{\pic}{Pic}
\newcommand{\sym}{\operatorname{Sym}}

\newcommand{\codim}{\operatorname{codim}}
\newcommand{\rk}{\operatorname{rk}}

\newcommand{\into}{\hookrightarrow}
\newcommand{\ladi}{\begin{lastadd}}
\newcommand{\ladf}{\end{lastadd}}
\newcommand{\lrei}{\begin{lastrem}}
\newcommand{\lref}{\end{lastrem}}
\newenvironment{lastadd}
{\cbstart\color{red}}
{\todo{red to remove}\cbend}
\newenvironment{lastrem}
{\cbstart\color{yellow}}
{\cbend}

\makeatletter
\renewcommand*\env@matrix[1][\arraystretch]{%
  \edef\arraystretch{#1}%
  \hskip -\arraycolsep
  \let\@ifnextchar\new@ifnextchar
  \array{*\c@MaxMatrixCols c}}
\makeatother

\author{Dominic Bunnett}
\address{
Institut f\"ur Mathematik\\
Technische Universit\"at Berlin\\
\email{bunnett@math.tu-berlin.de}
}
\author{Hanieh Keneshlou}
\address{
Max-Planck Institute for Mathematics in Sciences, Leipzig\\
\email{hanieh.keneshlou@mis.mpg.de}
}

\MSC{14Q10}
\title{Determinantal representations of the cubic discriminant}

\begin{document}
\maketitle
\begin{abstract}
We compute and study two determinantal representations of the discriminant of a cubic quaternary form. The first representation is the Chow form of the $2$-uple embedding of $\PP^3$ and is computed as the Pfaffian of the Chow form of a rank 2 Ulrich bundle on this Veronese variety. We then consider the determinantal representation described by Nanson. We investigate the geometric nature of cubic surfaces whose discriminant matrices satisfy certain rank conditions. As a special case of interest, we use certain minors of this matrix to suggest equations vanishing on the locus of $k$-nodal cubic surfaces.
\end{abstract}

\section*{Introduction}
Given a square system of homogeneous polynomial equations $f_0=\cdots= f_n=0$ of degrees $d_0,\ldots,d_n$ and in variables $x_0, \dots , x_n$, there is a unique  (up to a non-zero constant) polynomial $\res(f_0,\ldots, f_n)$, the \textit{resultant}, in the coefficients of the equations, whose vanishing is necessary and sufficient for the existence of non-trivial solutions for the system of equations. This is a polynomial invariant under the action of the special linear group $\SL(n+1)$. The resultant can be viewed as generalisation of the determinant in the case $d_0=\cdots=d_n=1$ where the resultant is nothing but the determinant of the coefficient matrix.

For a homogeneous polynomial $F$ of degree $d$, the \textit{discriminant}, denoted by $\Delta(F)$ is the resultant of first partial derivatives of $F$, that is
\begin{small}
\[\Delta(F) = \res(\frac{\partial F}{\partial x_0}, \cdots , \frac{\partial F}{\partial x_n}).\]
\end{small}
Therefore, the vanishing of $\Delta(F)$ assures the existence of singular points on the hypersurface $V(F)$ and is a polynomial in \begin{scriptsize}$\left( \begin{matrix} n+d \\d \end{matrix} \right)$\end{scriptsize} variables of degree \linebreak $(n+1)(d-1)^n$; see \cite[Example IX.1.6 (a)]{gkz}. We refer to \cite[Chapters 9 and 13]{gkz} for more details on the discriminant and its properties.

In this paper, we focus on determinantal representations of the discriminant of a quaternary cubic form which we refer to as the discriminant in the rest of the paper. In \cite{s61}, Salmon computed the discriminant of a cubic surface. This discriminant is a degree $32$ polynomial in the twenty coefficients of the general quaternary cubic form, although Salmon did not give it explicitly as the number of monomials is very large; it has at the very least 166,104 monomials \cite{kl19}. Note that Salmons original calculations contained a minor error, corrected by Edge \cite{e80}. Given the size, one requires different representations to be able to work with and study the discriminant. In this paper, we study two different ways the coefficients of a cubic can be organized into a matrix such that the determinant of that matrix is the discriminant.

The first representation we consider comes from the identification of the resultant of four quaternary quadrics with the Chow form of the Veronese embedding of $\PP^3$ in $\PP^9$ (Lemma \ref{lemma_res_chow}). The Chow form is a single polynomial defining the Chow divisor of this Veronese variety in the Grassmannian $\Gr(4,10)$ in Pl\"ucker coordinates. See \cite[Section III.2.B]{gkz} for a comprehensive introduction to Chow forms and Chow divisors. To compute this Chow form we use techniques of homological and exterior algebra developed in \cite{es03,efs03}.

Let us sketch the procedure. Let $\nu : \PP^m \into \PP^n$ be a Veronese embedding. The Chow form of a Veronese variety can be computed by finding a resolution of a sheaf on the Grassmannain supported on the Chow divisor and taking its determinant \cite[Chapter II]{mk76} and \cite[Appendix A]{gkz}. Such a sheaf is found by considering a sheaf on $\PP^n$ supported on $\PP^m$ and constructing a sheaf on the Grassmannian via the incidence correspondence \cite[Section 5]{es03}. However, up until the work Eisenbud and Schreyer \cite{es03}, computing these resolutions and their determinants had not been possible. In \cite{es03}, an effective method was given to compute the relevant resolutions, if one can find an Ulrich sheaf (Definition \ref{def_ulrich}) supported on $\nu(\PP^m) \subseteq \PP^n$. When the sheaf is Ulrich, the resolution is a two-term complex (the so-called Chow complex) and the determinant of the map in this complex is the Chow form raised to the power of the rank of the sheaf.

In the case of cubic surfaces, there exists an Ulrich sheaf and the map in the Chow complex is given by a $16 \times 16$ matrix with linear entries in $210$ Pl\"ucker coordinates whose Pfaffian provides the resultant. To retrieve the resultant from this matrix, one substitutes Pl\"ucker coordinates with the determinants of $4 \times 4$ matrices constructed from coefficients of the four quadrics.

The second representation is due to Nanson and provides the discriminant as the determinant of a $20\times 20$ matrix \cite{ns99}. This is a matrix with 16 rows of linear entries and 4 rows of quadratic entries. We investigate the geometry of the cubic surfaces lying on certain subspaces cut out by imposing rank conditions on this matrix. We provide evidence that certain minors of this matrix give equations vanishing on the locus of $k$-nodal cubics for $k$ up to $4$.

The paper is structured as follows. The first section deals with the computation of the discriminant by means of the resultant of four quaternary quadrics and the associated Chow form. We give a short, example driven exposition to the theory developed by Eisenbud and Schreyer, which we then utilise to compute the discriminant of a cubic quaternary form, which is given as the Pfaffian of a $16 \times 16$ matrix.

Section \ref{nanson} concerns the determinantal representation of the discriminant calculated by Nanson. Imposing rank conditions on this matrix, we look for stratification of the locus of singular cubic surfaces. In Section \ref{nodal}, we investigate the relationship between these strata and the locus of cubics with a given number of nodal singularities. 

Our results and conjectures rely on the computations performed by the computer algebra system \textit{Macaulay2} \cite{m2}, and using the supporting functions in \cite{bk}.
\vspace{-0.5cm}
\section*{Acknowledgements}
We would like to thank Michael Stillman and Frank-Olaf Schreyer for the invaluable conversations and help regarding the Macaulay2 computations. The work presented in this paper addresses the questions 4, 6, and 7 in \cite{bs} and we would like to thank Bernd Sturmfels for bringing these problems to our attention. We would also like to thank the anonymous referee whose comments improved the clarity of the material.
\vspace{-0.5cm}
\section{The discriminant as a Chow form}

In this section, we introduce all the necessary background material. We start with the construction of Chow assigning a \em Chow form \em to every subvariety in projective space and its extension to sheaves on projective space. We refer the reader to \cite[Chapter III]{gkz} for a comprehensive introduction to Chow forms of projective varieties and to \cite[Section 2]{mk76} for the extension to sheaves.

\subsection{Chow forms and Chow complexes}

Let $W$ be an $(n+1)$-dimensional vector space over a field $\mathbb{K}$, and $X\subset\PP^n=\PP(W)$ be a $k$-dimensional variety. The \em Chow divisor \em of $X\subseteq \PP^n$ is the hypersurface in the Grassmannian $\Gr = \Gr(k+1,n)$ of codimension $(k+1)$ linear subspaces of $\PP^n$, whose points are subspaces which meet $X$. The Chow divisor is denoted by $D_X \subseteq \Gr$ and the \em Chow form \em of $X$ is the defining equation of $D_X$ in Pl\"ucker coordinates \cite[Proposition III.2.1]{gkz}. Explicitly, let
\[\mathbb{I}_{k+1} = \{(x,L) \in \PP^n \times \Gr \, \, | \, \, x \in L\} \subseteq \PP^n \times \Gr\]
be the incidence correspondence and consider the projection maps
\[\PP^n \xlongleftarrow{\pi_1} \mathbb{I}_{k+1} \xlongrightarrow{\pi_2} \Gr,\]
then $D_X = \pi_2(\pi_1^{-1}(X))$. One can check that this is indeed a hypersurface \cite[Proposition III.2.2]{gkz}. The Chow divisor of a $k$-cycle $\sum_in_i\cdot [X_i]$ is defined to be $\sum_i n_i \cdot D_{X_i}$.

More generally, let $\cF$ be a sheaf on $\PP^n$ whose support $X$ is $k$-dimensional, and consider the associated $k$-cycle 
\[\sum_{Y}\length(\cO_{\PP^n,Y}\otimes \cF)\cdot [Y],\]
where the sum is taken over all $k$-dimensional components of $X$. The Chow divisor of $\cF$ is defined to be the corresponding sum of Chow divisors. For instance, if $\cF=\cO_{X}$, then the Chow divisor of $\cF$ is the Chow divisor of $X$; more generally, if $\cF$ is a vector bundle of rank $r$ on $X$, the Chow divisor of $\cF$ is r times that of $X$. As before, the Chow form of $\cF$ is the defining equation of the Chow divisor.

The question is then given a sheaf, supported on an irreducible variety $X$, how can one compute the Chow form.

Let $\cG := (\pi_2)_*\pi_1^*\cF$ which is a sheaf supported on $D_X$. In a letter to Mumford \cite{gro}, Grothendieck observed that there exists a locally free complex $\cC$ and a quasi-isomorphism
\[\cC \simeq \R\pi_{2*}(\pi_1^*\cF).\]
In particular, $\Hb^{0} \cC \cong \cG$ and $\Hb^{i}\cC \cong \R^i\pi_{2*}(\pi_1^*\cF)$. Such a complex is called a \em Chow complex \em of $\cF$ and the Chow form of $\cF$ is the determinant of $\cC$. The notion of the determinant of a complex was first introduced by Grothendieck for exactly this purpose and a full account of the construction is worked out by Knudsen and Mumford in \cite{mk76}.

Therefore, the computation of a Chow form turns to the computation of the determinant of a Chow complex. In the following section, we recall an explicit and canonical construction of Chow complexes due to Eisenbud and Schreyer \cite{es03} using the notion of Tate resolutions.

\subsection{Tate resolutions and Ulrich sheaves}
We recall the construction of Tate resolutions for a coherent sheaf. For a detailed reference see \cite{efs03}.

As before, let $W$ be an $(n+1)$-dimensional vector space over $\mathbb{K}$ with basis $x_0,\ldots,x_n$ and let $S=\sym W=K[x_0,\ldots,x_n]$ be the symmetric algebra of $W$. Moreover, let $V=W^\vee$ be the dual vector space with basis $e_0,\ldots,e_n$, and $E=\bigoplus_{i=0}^{n}\bigwedge^{i}V$ be the exterior algebra over $V$. The gradings on $S$ and $E$ are given by $\deg(x_i)=1$ and $\deg(e_i)=-1$.

There is a pair of adjoint functors, the BGG correspondence \linebreak \cite{bgg}, between the categories of complexes of finitely generated graded modules over $S$ and $E$
\[\cplx(S) \xrightleftharpoons[\quad\, L \,\quad]{R} \cplx(E).\]
To a graded $S$-module $M=\bigoplus_{d} M_d$ we associate the following complex of free $E$-modules
\begin{equation}\label{eq:R-diff}
R(M): \cdots\longrightarrow E \otimes M_d \xlongrightarrow{\phi_d} E\otimes M_{d+1}\longrightarrow\cdots
\end{equation}
with differential map $\phi_d=\sum_{i=0}^{n}e_i\otimes x_i$. Similarly, for a graded $E$-module $P=\bigoplus_{j} P_j$, one builds the following complex of free $S$-modules
$$ L(P): \cdots\longrightarrow S\otimes P_j \xlongrightarrow{\phi_j} S\otimes P_{j-1}\longrightarrow\cdots $$
where $\phi_j=\sum_{i=0}^{n}x_i\otimes e_i$. 

Now, let $\mathcal{F}=\widetilde{M}$ be sheafification of a finitely generated graded $S$-module and let $s=\reg(M)$ denote the Castelnouvo-Mumford regularity of $M$. With the above notation, set
\[ P^s:=\ker(E\otimes M_s \longrightarrow E\otimes M_{s+1})\]
and let 
\[ \cdots \longrightarrow T^{s-2}\longrightarrow T^{s-1}\longrightarrow P^s\longrightarrow 0\]
be a relative projective resolution of $P^s$. Adjoining this complex with the truncated complex $R(M)_{>s}$, one obtains \textit{a Tate resolution} of $\mathcal{F}$:
\[ T(\mathcal{F}):\quad  \cdots \longrightarrow T^{s-2}\longrightarrow T^{s-1}\longrightarrow E\otimes M_s \longrightarrow E\otimes M_{s+1}\longrightarrow \cdots \]
It is proved in \cite[Theorem 4.1]{efs03} that the $k$-th term in the Tate resolution is isomorphic to
\[T^k(\mathcal{F})=\bigoplus_{j=0}^n\Hb^j(\mathcal{F}(k-j))\otimes E(j-k).\]
Here one must be careful of possibly confusing notation; $\cF(i)=\cF \otimes \cO_{\PP^n}(i)$ and $E(i)$ represents a shift in the grading by $i$. Tate resolutions of a coherent sheaf $\cF$ can be defined in more general context (see \cite[Section 3]{es08}) and are closely related to the \em Beilinson monad \em \cite{b78}. Indeed, one can obtain the Beilinson monad, and hence the sheaf $\cF$, from Tate resolutions \cite[Theorem 6.1]{efs03}.

Let us explain how to extract a Chow complex from the Tate resolution. We can define an additive functor $\uu_{k+1}$ from graded free modules over $E$ to locally free sheaves on $\Gr$ by sending $E(p)$ to $\uu_{k+1}(E(p)) = \bigwedge^p U$, where $U \subseteq \cO_{\Gr} \otimes W $ is the tautological subbundle of rank $k+1$. A map $\eta : E(q) \rightarrow E(q-p)$ is represented by an element $\alpha \in \bigwedge^{p} V$ which induces a map $\bigwedge^p W \rightarrow\mathbb{K}$ by contraction and from that one obtains the map $\uu_{k+1}(\eta) : \bigwedge^q U \rightarrow \bigwedge^{q-p}U$.

Note that $\uu_{k+1}(E(p)) = 0$ unless $0 \leq p \leq \rk U = k+1$ and thus the complex
$$\uu_{k+1}(\cF) := \uu_{k+1}(T(\cF))$$
is a bounded complex of locally free sheaves. Then we can apply \cite[Theorem 1.2]{es03}, which says that
$$\uu_{k+1}(\cF)\simeq \R\pi_{2*}(\pi_1^*\cF)$$
and thus $\uu_{k+1}(\cF)$ gives a chow complex of $\cF$.

The most accessible formulas for the Chow form occur when the complex has only one non-trivial map $\Psi$:
\[ \cdots \longrightarrow 0 \longrightarrow 0\longrightarrow C^{-1}\xlongrightarrow{\Psi}C^0\longrightarrow 0\longrightarrow \cdots .\]
In this case, the Chow form of $\cF$ is given by the determinant of $\Psi$. Moreover, if $\cF$ has rank $2$ and the map $\Psi$ is skew-symmetric (in the sense of \cite[Section 3.1]{es03}), the Pfaffian of $\Psi$ computes the Chow form of the support of $\cF$. This leads to the notion of Ulrich sheaves, which possess the desired two term Chow complex. These are sheaves with no intermediate cohomologies:

\begin{definition}\label{def_ulrich}
A coherent sheaf $\cF$ with $k$-dimensional support on projective space $\PP^n$ is called an Ulrich sheaf if it satisfies the following three vanishing conditions

\begin{enumerate}
\item $\Hb^i\cF(j)=0$ for $1\leq i \leq k-1$ and for all $j$,
\item  $\Hb^k\cF(l)=0$ for all $l \geq -k$, and
\item $\Hb^0\cF(j)=0$ for $j<0$.
\end{enumerate}
\end{definition}
One can collect the dimension of all cohomology groups of all twists of a sheaf in a so-called Betti diagram. In particular, the Betti diagram of an Ulrich sheaf $\cF$ is of the following form:
\begin{center}\small{
 \begin{tabular}{c c c c c c c c}
$\cdots$ & $h^k\cF(-k-3)$ & $ h^k\cF(-k-2) $& $h^k\cF(-k-1)$ & $0$ & $0$& $ 0 $ &\\  [0.5ex]
    & $0$ & $ 0 $&$0$ & $0$ & $0$& $ 0 $&\\ 
    &$\vdots$ & $\vdots$ & $ \vdots $& $\vdots$ & $\vdots$ &$\vdots$ & \\
    &$ 0 $ & $ 0 $ & $ 0 $ & $0$ & $0$ &$0$& \\
    & $0$ & $0$ & $ 0 $& $h^0\cF$ & $h^0\cF(1)$ &$h^0\cF(2)$ & $ \cdots $ \\
\end{tabular}}
\end{center}
It follows from the Betti diagram above that the Chow complex $\uu(\cF)$ is given by a $h^0\cF \times h^k\cF(-k-1)$ matrix whose entries are linear forms in the Pl\"ucker coordinates on $\Gr$.

\begin{example}\label{ex_ulrich_curves}
Consider the $d$-uple Veronese embedding $\nu : \PP^1 \hookrightarrow \PP^d$. Let $\cL = \cO_{\PP^1}(-1)$, we claim that $\cF := (\nu_*\cL)(1)$ is an Ulrich sheaf for $C = \nu(\PP^1)$. Indeed, using the projection formula we have that $\cF(j-i) = \nu_*(\cO_{\PP^1}(d-id+jd-1))$. Then using the fact that $h^0(\PP^1,\cO_{\PP^1}(l)) = l+1$ and by applying Serre duality, we have that the Betti table of $\cF$ is as follows.
\begin{center}
 \begin{tabular}{c|c c c c c c c c c}
  & $\cdots$ & $-3$ & $ -2 $& $-1$&$0$ & $1$ &  $2$&$ \ldots $\\
  [0.5ex]
 \hline
$1$ & $***$ & $3d$ & $ 2d $& $d$&$.$ & $.$ &  $.$& $\ldots $\\ 
$0$ & $\cdots$ & $.$ & $ .$& $.$&$d$ & $2d$ &  $3d$& $***$ \\ 
\end{tabular}
\end{center}
Note that in the above table (and in the rest of the article), we have replaced 0 by  ` $\cdot$ ' for clarity. Thus it follows that $\cF$ is an Ulrich sheaf as claimed.
\end{example}

\subsection{The cubic discriminant as a Chow form}\label{matrix}

Let $k$ and $d$ be two positive integers, $W = H^0(\PP^k , \cO_{\PP^k}(d))$ and $X = \PP^k \subseteq \PP(W) = \PP^n$ be the $d$-uple Veronese embedding. A hyperplane in $\PP^n$ may be regarded as a form of degree $d$ on $X$. Moreover, if $k+1$ hyperplanes have common zeros on $X$, then the corresponding forms have a common zero. If the $k+1$ planes do have a common zero on $X$, we may intersect them and we obtain an element of the Chow divisor $D_X \subseteq \Gr$. This proves the following:

\begin{lemma}\label{lemma_res_chow}
The Chow form of the Veronese embedding of $\PP^k \subseteq \PP^n$ is the resultant of $k+1$ forms of degree $d$ in $k+1$ variables.
\end{lemma}
We consider the example of the twisted cubic curve in $\PP^3$ to demonstrate the heavy machinery developed in the previous sections. In this case, the resultant corresponding to this Chow form can be computed using the  classical method of the Sylvester matrix and we compare the two different approaches. All the computations done below, including the comparison between the two methods, were carried out in Macaulay2 and can be found in \cite{bk}. 

\begin{example}
Consider the vector space of binary cubic forms $W = H^0(\PP^1,\cO_{\PP^1}(3))$ and let $C \subset \PP(W) \cong \PP^3$ be the twisted cubic curve; that is, the image of the $3$-uple embedding of $\PP^1$ given by
\begin{align*}
 \nu:   \PP^1& \hookrightarrow  \PP^3\\
 (x_0:x_1)& \mapsto (x_0^3:x_0^2x_1:x_0x_1^2:x_1^3).
\end{align*}
Let $\mathcal{L}=\cO_{\PP^1}(-1)$ such that $\cF=(\nu_*\mathcal{L})(1)$ is an Ulrich sheaf on $C$ by Example \ref{ex_ulrich_curves}. Subbing $d=3$ into the table from Example \ref{ex_ulrich_curves} we get
\begin{center}
 \begin{tabular}{c|c c c c c c c c c}
  & $\cdots$ & $-3$ & $ -2 $& $-1$&$0$ & $1$ &  $2$&$ \ldots $\\
  [0.5ex]
 \hline
$1$ & $***$ & $9$ & $ 6 $& $3$&$.$ & $.$ &  $.$& $\ldots $\\ 
$0$ & $\cdots$ & $.$ & $ .$& $.$&$3$ & $6$ &  $9$& $***$. \\ 
\end{tabular}
\end{center}
One sees immediately that $\cF$ is 0-regular. Therefore the Tate resolution of $\cF$ is of the form
\[T(\cF):\quad \cdots  \longrightarrow E(3)^{6}\xlongrightarrow{\phi_{-2}}E(2)^{3} \xlongrightarrow{\phi_{-1}}E^{3} \xlongrightarrow{\phi_0} E(-1)^{6} \longrightarrow \cdots .\]
where $E = \bigwedge W^\vee$ is the exterior algebra associated to the dual space of $W$. The map 
\[\phi_0 : H^0(\PP^1,\cO_{\PP^1}(2)) \otimes E \longrightarrow H^0(\PP^1,\cO_{\PP^1}(5)) \otimes E(-1)\]
is the differential map of the $R$-complex $(\ref{eq:R-diff})$, and is calculated explicitly via the multiplication map
\[H^0(\PP^1 , \cO_{\PP^1}(2))\otimes W\longrightarrow H^0(\PP^1 , \cO_{\PP^1}(5)).\]
Let $S=\mathbb{K}[x_0,x_1]$ be the homogeneous coordinate ring of $\PP^1$ so that
\begin{align*}
H^0(\PP^1,\cO_{\PP^1}(2)) : &\quad x_0^2,x_0x_1,x_1^2\\
H^0(\PP^1,\cO_{\PP^1}(3)) : & \quad x_0^3,x_0^2x_1,x_0x_1^2,x_1^3\\
H^0(\PP^1,\cO_{\PP^1}(5)) : & \quad x_0^5,x_0^4x_1, \dots , x_1^5
\end{align*}
are the monomial bases for the vector spaces in the multiplication map. Denote by $e_0, \dots , e_3$ the basis of $W^\vee$ dual to the monomial basis. Writing the matrix of $\phi_0$ with respect to the basis given by the $e_i$ yields the following:
$$\phi_0=
\left( \begin{array}{rrrrrr}
e_0 & 0 & 0  \\
e_1 & e_0 & 0  \\
e_2 & e_1 & e_0  \\
e_3 & e_2 & e_1  \\
0 & e_3 & e_2  \\
0 & 0 & e_3 \\
\end{array}\right).$$
To compute $\phi_{-1}$, one computes the syzygy matrix of $\phi_0$. This can be done either by hand or quickly in Macaulay2 (see \cite{bk}):
\[\phi_{-1}=
\begin{pmatrix}
 e_0e_3& e_0  e_2 & e_0e_1 \\
e_1e_3& e_1e_2+e_0e_3\ & e_0e_2\\
e_2e_3& e_1 e_3\ & e_0e_3 \\
\end{pmatrix}.\]
If we interpret the $2$-forms $e_ie_j$'s as the Pl\"ucker coordinates
 $$[ij]:=e_i\wedge e_j,\quad i,j\in \lbrace 0,\ldots,3 \rbrace$$
 of the Grassmannian $\Gr(2,W)$, under this identification, the map $\phi_{-1}$ defines the non-zero matrix in the two-term Chow complex of $\cF$ 
\[0\longrightarrow \cO_{\gr}(-1)^{3}\longrightarrow \cO_{\gr}^{3}\longrightarrow 0,\]
obtained by applying the functor $\uu_2$ to the Tate resolution $T(\cF)$, explicitly given by
\[\uu_2(\phi_{-1})=
\begin{pmatrix}[]
[03] & [02] & [01] \\
[13] & [12]+[03] & [02]\\
[23] & [13] & [03] \\
\end{pmatrix}.
\]
The determinant of this $3\times 3$ matrix gives the Chow form of $C$. In particular,
the resultant of two binary cubic forms
\begin{align*}
f&=a_0x_0^3+a_1 x_0^2x_1+a_2 x_0x_1^2+a_3x_1^3\\
g&=b_0x_0^3+b_1 x_0^2x_1+b_2 x_0x_1^2+b_3x_1^3    
\end{align*}
is obtained via the substitution of the Pl\"ucker coordinates with
\[[ij]=\det \left( \begin{array}{rrrr}
a_i & b_i \\
a_j& b_j\\
\end{array}\right), \quad i,j\in \lbrace 0,\ldots,3 \rbrace\]
in the determinant of the matrix. This gives us the following expression of the Chow form of $C$:
\[{a}_{0}^{3}{b}_{3}^{3}-{a}_{0}^{2}{a}_{1}{b}_{2}{b}_{3}^{2}+{a}_{0}{a}_{1}^{2}{b}_{1}{b}_{3}^{2}-{a}_{1}^{3}{b}_{0}{b}_{3}^{2}+{a}_{0}^{2}{a}_{2}{b}_{2}^{2}{b}_{3}+ \cdots -{a}_{3}^{3}{b}_{0}^{3}.
\]
Comparing this expression with the determinant of the Sylvester matrix
\[\syl(f,g) = \begin{pmatrix}
 a_0 & a_1 & a_2 & a_3 & 0 & 0 \\
 0 & a_0 & a_1 & a_2 & a_3 & 0\\
 0& 0 & a_0 & a_1 & a_2 & a_3 \\
 b_0 & b_1 & b_2 & b_3 & 0 & 0 \\
 0 & b_0 & b_1 & b_2 & b_3 & 0\\
 0& 0 & b_0 & b_1 & b_2 & b_3 \\
 
\end{pmatrix},\]
and indeed the two results agree. See (\cite[Example]{bk}) for the explicit computations of this example in Macaulay2.
\end{example}

Now, as our main case of study, the Chow form of the 2-uple embedding $\nu : \PP^3 \hookrightarrow \PP^9 $ is the resultant of four quadrics in $4$ variables. In this way, the discriminant of a cubic surface can be computed via the Chow form of some Ulrich bundles on $\PP^9$ supported on $\nu(\PP^3)$.

Let $\cE$ be the \em Null-Correlation bundle \em on $\PP^3$, defined to be the cokernel of a monomorphism $\pi$ fitting into the short exact sequence
\[0\longrightarrow \mathcal{O}_{\PP^3}(-1)\xlongrightarrow{\pi}\Omega_{\PP^3}(1)\longrightarrow \mathcal{E}\longrightarrow 0.\quad (\ast)\]

Let $\nu : \PP^3 \hookrightarrow X\subset\PP^9$ be the 2-uple embedding map. We define the sheaf $\cG := \cE(2)$ and consider the sheaf $\cF := \nu_*\cG$ on $\PP^9$ supported on $X$.

\begin{theorem}\cite[Theorem 5.11]{es03}
Keep the notation of above. Then $\cF$ is, up to automorphisms of $\PP^3$, the unique rank 2 Ulrich bundle on the $2$-uple embedding of  $\PP^3$. 
\end{theorem}

\begin{remark}
In the best case scenario, we would find a rank 1 Ulrich sheaf on $\PP^9$ supported on $\PP^3$, however there can be no such rank 1 bundle. Indeed, by \cite[Theorem III.5.1]{h77}, line bundles on $\PP^3$ have no intermediate cohomology; that is, $H^i(\PP^3,\cL) = 0$ for $0<i<3$ for every $\cL \in \pic(\PP^3)$. However, \cite[Proposition 5.1]{es03} implies that if $\nu_*\cL$ were an Ulrich sheaf for some $\cL \in \pic(\PP^3)$, then $H^1(\cL(-1)) \neq 0$.
\end{remark}

The Betti diagram of the Tate resolution of $\mathcal{G}$ can be found in Figure \ref{fig:bettiTable}.
\begin{figure}[ht]
\begin{center}
 \begin{tabular}{c|c c c c c c c c c c c c}
  & $\cdots$ & $-3$ & $ -2 $& $-1$&$0$ & $1$ &  $2$& $ 3 $ & $4$& $5$ & $6$&$ \ldots $\\ [0.5ex]
 \hline
$3$ & $***$ & $64$ & $ 35 $& $16$ & $5$ & $.$& $ . $ &$ . $&$ . $&$ . $&$ . $&$ \ldots $\\ 
 $2$& $\ldots$ & $.$ & $ . $&$.$ & $.$ & $1$& $ . $& $ . $& $ . $& $ . $&$ . $&$ \ldots $\\ 
$1$ & $\ldots$ &$.$ & $.$ & $ . $& $.$ & $.$ &$1$ &$ . $&$ . $&$ . $&$ . $& $\ldots $\\  
$0$ & $\ldots$ &$.$ & $.$ & $ . $& $.$ & $.$ &$.$ &$ 5 $&$ 16 $&$ 35 $&$ 64 $&$ *** $\\
\end{tabular}
\caption{Betti diagram of the Tate resolution of $\cG$}
\label{fig:bettiTable}
\end{center}
\end{figure}
Here the integer in the $k$-th column and the $j$-th row stands for $\hsm^j\mathcal{G}(k-j)$.

Using the table in Figure \ref{fig:bettiTable}, one checks that $\mathcal{F}$ is $0$-regular. Therefore, the Tate resolution of $\mathcal{F}$ is of the form 
\[T(\cF):\quad \cdots  \longrightarrow E(4)^{16} \longrightarrow E^{16} \xlongrightarrow{\phi_0} E(-1)^{64} \longrightarrow \cdots .\]
where $E$ is the exterior algebra associated to the dual vector space of $W=H^0(\PP^3 , \cO_{\PP^3}(2))$. Applying the functor $\uu_4$ to $T(\cF)$, we obtain the two-term complex
\[0\longrightarrow \cO_{\gr}(-1)^{16}\xlongrightarrow{\Psi} \cO_{\gr}^{16}\longrightarrow 0,\]
whose Pfaffian gives the Chow form of the $2$-uple embedding of $\PP^3$ in $210$ Pl\"ucker coordinates of $\Gr(4,W)$.

We now go through the computation of the map $\Psi$. Following the approach of the previous example, the map $\phi_0$ here comes from the multiplication map
\[\Hb^0\mathcal{E}(2)\otimes W\longrightarrow \Hb^0\mathcal{E}(4).\]
One can use the short exact sequence $(\ast)$ and the twists of the Euler sequence to compute a basis for the spaces global sections of $\cE(d)$ and from that the map $\phi_0$. However, as an alternative and easier way in this case, we construct a Tate resolution of $\cG$ in a broader sense, and then we compute the map $\phi_0$ from that. 

Let $E'$ be the exterior algebra associated to the $4$-dimensional vector space $W'=H^0(\PP^3 , \cO_{\PP^3}(1))$. We start with a non-degenerate quadratic form in $E'$ defining the middle $1\times 1$ matrix in a Tate resolution of $\cG$. Resolving such a form gives the tail of the Tate resolution, so that $\phi_0$ can be computed from the composition of the two linear maps 
\[T(\cG): \cdots\longrightarrow E'(4)^{16}\longrightarrow E'(5)^{35}\longrightarrow E'(6)^{64}\longrightarrow \cdots.\]
The first syzygy matrix of $\phi_0$ (possibly after a base change) gives the desired $16\times 16$ skew-symmetric matrix $\Psi$ with linear entries in Pl\"ucker coordinates. See (\cite[Resultant1]{bk}) for the explicit computation of this matrix, and \cite{bkb} for the print out of it in the pdf file \texttt{Matrix16}. See Figure \ref{fig:matrix} for the top left hand corner of the matrix.

The Pfaffian of $\Psi$ is the resultant of four quaternary quadratics $Q_i$, $i=1,2,3,4$ with $10$ coefficients in $a,b,c,d$ all sorted in a fixed ordering of the $10$ squared monomials 
$$x_0^2,\ x_0x_1,\ x_0x_2,\ x_0x_3,\ x_1^2,\ x_1x_2, x_1x_3,\ x_2^2,\ x_2x_3,\ x_3^2,$$
after substituting the Pl\"ucker coordinates with  
\[[ijkl]=\det 
\left( \begin{array}{rrrr}
a_i & b_i & c_i & d_i \\
a_j & b_j & c_j & d_j \\
a_k & b_k & c_k & d_k \\
a_l & b_l & c_l & d_l \\
\end{array}\right).\]

To compute the discriminant of a cubic quaternary form, we take each of the partial derivatives as one of the quadrics. The determinant of the $16 \times 16$ matrix is of degree 64, hence the Pfaffian is of degree 32, the expected degree of the discriminant as computed by Salmon \cite{s61}. See \cite[\texttt{Matrix20.m2}]{bkb} for the presentation of this matrix in terms of $20$ coefficients of the general cubic form.
\begin{figure}[ht]
    \centering \tiny
$  \begin{pmatrix}
  \begin{matrix}[1.5]
  0 & 4[3689] & 2[1489] + 4[5689] & 2[3679] + [0489] + 4[2689] \\
-4[3689] & 0 & -2[3679] + [0489] - 4[3589] & 2[0189] + 4[2389] \\
-2[1489] - 4[5689] &  2[3679] - [0489] + 4[3589] & 0 & 4[3678] + 2[3579] + 2[2679] + 4[2589] \\
-2[3679] - [0489] - 4[2689] & -2[0189] - 4[2389]&  -4[3678] - 2[3579] - 2[2679] - 4[2589] & 0 \\
  \end{matrix} &\cdots \\
   \vdots  & \ddots \,\,
  \end{pmatrix}$
    \caption{}
    \label{fig:matrix}
\end{figure}
%%%%%%%%%%%%%%%%%%%%%%%%%%%%%%%%%%%%%%%%%%%%%%%%%%%%%%%%%%%%%%%%%%%
%%%%%%%%%%%%%%%%%% SECTION 2: NANSON'S REP  %%%%%%%%%%%%%%%%%%%%%%%
%%%%%%%%%%%%%%%%%%%%%%%%%%%%%%%%%%%%%%%%%%%%%%%%%%%%%%%%%%%%%%%%%%%

\section{Nanson's determinantal representation}\label{nanson}

In \cite{ns99}, Nanson gave a method to represent the resultant of four quaternary quadrics as the determinant of a $20\times20$ matrix. We recall the construction briefly here. 

Let $Q_1,\ldots,Q_4$ be four quadratics in four variables  $x_0,x_1,x_2,x_3$. By multiplying each quadric $Q_i$ with each of the variables in turn we get sixteen cubics $f_1,\ldots,f_{16}$. Furthermore, one  has four cubic polynomials named $f_{17},\ldots,f_{20}$ as the first partial derivatives of the determinant of the Jacobian matrix. Eliminating the twenty cubic monomials $x_0^3,x_0^2x_1,\ldots,x_3^3$ from the twenty cubics, Nanson proved that the determinant of the coefficient matrix $M$
$$M\cdot\left( \begin{array}{rrrr}
x_0^3\  \\
x_0^2x_1 \\
\vdots\ \ \\
x_3^3\\
\end{array}\right)= \left( \begin{array}{rrrr}
f_1  \\
f_2 \\
\vdots\ \\
f_{20}\\
\end{array}\right)$$
is the resultant of $Q_1,\ldots,Q_4$; a homogeneous polynomial of degree $8$ in the coefficients of each quadric. In particular, the discriminant of a general quaternary cubic form arises as the determinant of a $20\times 20$ matrix $M$ with $16$ rows of linear and $4$ rows of quadratics entries in $20$ variables.

One obtains a stratification of the locus of singular cubic surfaces by imposing rank condition on $M$. Explicitly, define $V_k \subset \PP^{19}$ to be the variety cut out by $(20-k+1)\times (20-k+1)$ minors of $M$. Then $V_1$ is the discriminant hypersurface. We would like to know about the geometry of the higher codimension strata; both of the strata themselves and the geometry of the cubic surfaces they contain. The latter is precisely the Question 6 from the \em cubic surfaces \em project \cite{bs}.\\

\noindent
\textit{Question 6: What is the geometric description of the cubic surfaces lying on $V_k$?}\\

\noindent
We will deal with this question in the next section.

\subsection{The locus of k-nodal cubics}\label{nodal}
Let $N_k\subset \PP^{19} $ denote the closure of the locus of $k$-nodal cubic surfaces. One can show via direct computation using Chern classes that $\codim N_k=k$ for $1\leq k\leq4$ and $N_k=\emptyset$ for $k>4$.

Note that $N_1$ is the entire discriminant hypersurface $V_1$, indeed since $V_1$ is an irreducible hypersurface, $\codim N_1 = 1$ and $N_1 \subseteq V_1$, we have $N_1 = V_1$. Moreover, \cite[Theorem I.1.5]{gkz} states that surfaces on the smooth locus of the discriminantal locus $V_1$ possess a unique singularity. Thus the general element of $V_1$ is a cubic surface with a single node.

Let $d_k=\deg N_k$ denote the degree of the locus $N_k$, interpreted as the number of $k$-nodal cubic surfaces in a general $k$-dimensional linear system of cubics. Rennemo \cite{rn17} has shown that the number of hypersurfaces with a certain type of singularity in an appropriate dimensional linear system is always governed by a polynomial in Chern classes of the associated line bundle. Relevant to the existence of the polynomial is the recent work of Y-J. Tzeng \cite{tz17}. In particular, Vaisenscher \cite{vs03} described these polynomials explicitly in a $k$-dimensional family of hypersurfaces with $k\leq 6$ ordinary double points. Evaluating his formulas in the particular case of cubic surfaces provides that $d_k$ is
$$32,\ 280,\ 800,\ 305,$$
for $ 1\leq k\leq4$, respectively. Now, a question one may ask towards understanding the geometry of such a locus is:\\

\noindent
\textit{Question 7: Can we find a low degree polynomial vanishing on $N_k$?}\\

In line with question 6, one can use a computer algebra system to build the matrix $M$ explicitly, however computing the determinant and minors of this matrix is still a heavy and unreachable task for current computer algebra systems. 

We take manageable steps towards this problem by considering the matrix modulo ideals generated by linear elements, of varying dimension. This corresponds to intersecting the varieties $V_k$ with low-dimensional linear subspaces. This sheds some light on geometry of cubic surfaces satisfying certain rank conditions.

As a first step, we consider the restriction to a general plane $\PP^2\subset \PP^{19}$.
\begin{proposition}
In the above notation, it holds that $\codim V_2=2$ and $\codim V_k\geq 3$ for $k\geq 3$. 
\end{proposition}

\begin{proof}
Our computations (see \cite[Nanson2]{bk}) show that $V_2 \cap \PP^2$ is a finite set of points and that $V_k \cap \PP^2 = \emptyset$ for $k\geq 3$. Thus by the dimension formula we can conclude that $\codim V_2=2$ and $\codim V_k\geq 3$ for $k\geq 3$.
\end{proof}

In this case, the discriminant is a plane curve $\Gamma=V_1\cap \PP^2=V(\dt\ M)\subset \PP^2$ of degree $32$ whose singular locus is of degree $520$.

It follows from \cite[Theorem I.1.5]{gkz} and Vaisenscher's computation of the degree of the binodal locus $N_2$ that this singular locus contains $280$ nodes corresponding to the binodal cubics. Moreover, one can check that the number of conditions imposed by a cusp on a cubic surface is the same as the number of the conditions for two nodes, so that the locus of binodal and cuspidal cubic surfaces are of codimension $2$ inside the space of cubic surfaces.  As a cusp contributes two times to the degree of the singular locus, we have $120=520-280/2$ remaining points in the singular locus corresponding to the cuspidal cubics (one can alternatively use the techniques of the Chern class computations for this enumerative problem to prove that in a general net of cubic surfaces there are exactly $280$ and $120$ fibers corresponding to the binodal and cuspidal cubic surfaces, respectively).

On the other hand, our computations show the intersection $V_2 \cap \PP^2$
consists of exactly 400 points. Comparing this to the equality $400=280+120$ suggests that these points are the singular points of the discriminant curve and indeed they do agree (\cite[Nanson2]{bk}).
We are lead to conjecture:

%By \cite[Theorem I.1.5]{gkz}, we know that the binodal locus lies in the singular locus of the discriminantal locus, that is $N_2 \subseteq V_1^{\text{sing}}$, thus we can conclude that
%\[N_2 \cap \PP^2 \subseteq V_2 \cap \PP^2.\]

\begin{conjecture}\label{conj_N2}
The locus $V_2$ has at least two irreducible components parameterising binodal and cuspidal cubic surfaces. In particular, we have $N_2\subset V_2$ and the $19\times 19$ minors of $M$ vanish on $N_2$.
\end{conjecture}

Let us present some further evidence for Conjecture \ref{conj_N2}. Let $X = V(F)$ with $F \in \mathbb{K}[x_0,x_1,x_2,x_3]_{3}$ be a cubic surface with a singularity. After a change of coordinates we may assume that the singular point of $X$ is at $p=(1:0:0:0)$. This gives us a normal form for cubic surfaces with a singularity at $p$:
\[F = F_3(x_1,x_2,x_3) + x_0F_2(x_1,x_2,x_3),\]
where $F_3$ is a cubic and $F_2$ is a conic in the variables $x_1,x_2,x_3$. Such equations make up a 15 dimensional linear subspace of $|\cO_{\PP^3}(3)| \cong \PP^{19}$ whose generic element is a smooth everywhere except $p$. Indeed if $F_3$ and $F_2$ are smooth, $X$ has a single nodal singularity at $p$.

Assuming that $X$ has two singularities, we may change coordinates such that they are at position $p = (1:0:0:0)$ and $q = (0:0:0:1)$. We may apply the same normal form as above, but with respect to $q$:
\[F = G_3(x_0,x_1,x_2) + x_3G_2(x_0,x_1,x_2),\]
where $G_3$ is a cubic and $G_2$ is a conic in the variables $x_0,x_1,x_2$. Combining the two normal forms we get
\begin{align*}
    F = a_0\cdot x_0x_1^2 & + a_1\cdot x_0x_1x_2 + a_2\cdot x_0x_1x_3 + a_3\cdot x_0x_2^2 + a_4\cdot x_0x_2x_3 + a_5\cdot x_1^3+\\& + a_6\cdot x_1^2x_2 + a_7\cdot x_1x_2^2 + a_8\cdot x_1x_2x_3 + a_9\cdot x_1^2x_3 + a_{10}\cdot x_2^3 + a_{11}\cdot x_2^2x_3. 
\end{align*}
Such equations make up a $\PP^{11} \subseteq \PP^{19}$, denoted by $N_2^{\text{nf}} = \PP^{11}$, whose general element is a cubic surface with nodes at $p$ and $q$. Using Macaulay2 one can check the following:

\begin{proposition}
The matrix of minors defining $V_2$ vanish on the binodal cubic surfaces in normal form. That is,
\[N_2^{\text{nf}} \subseteq V_2.\]
\end{proposition}
\begin{proof}
This is proved by direct computation. See \cite[Secion 3]{bk}.
\end{proof}

\begin{remark}
The computation took around 35 hours and is at the limit of what is possible to symbolically compute.
\end{remark}

\noindent
Now considering the drop by dimension when we restrict to minors of smaller size, the pattern suggests that:
\begin{conjecture}
For $1\leq k\leq4$, the $(20-k+1)\times(20-k+1)$ minors of $M$ vanish on $N_k$.
\end{conjecture}
Proceeding with this approach, the next step would be the restriction to a general subspace $\PP^3\subset \PP^{19}$ for $V_3$. However, there are $36,100$ minors of Nanson's matrix $M$, and this becomes too big for computer systems.

\end{document}